\documentclass[english]{article}
\usepackage{amsmath,amssymb,latexsym,amsthm,mathrsfs,appendix}
\usepackage{color}
\usepackage{scalerel}
\usepackage{graphicx}
\usepackage{caption}
\usepackage{subcaption}
\usepackage{mathtools}

\date{}

\def\theenumi{\arabic{enumi}}

\def\theenumii{\alph{enumii}}
\def\p@enumii{\theenumi.}

\def\theenumiii{\arabic{enumiii}}
\def\p@enumiii{(\theenumi)(\theenumii)}

\def\p@enumiv{\p@enumiii.\theenumiii}
\newtheorem{theorem}{Theorem}[section]

\newtheorem{corollary}[theorem]{Corollary}

\newtheorem{proposition}[theorem]{Proposition}
\newtheorem{lemma}[theorem]{Lemma}

\theoremstyle{definition}

\newtheorem{definition}[theorem]{Definition}
\newtheorem{remark}[theorem]{Remark}

 \makeatother

\usepackage{babel}

\begin{document}
\title{Sub-elliptic diffusions on compact groups via Dirichlet form perturbation}
\author{Qi Hou\thanks{%
Partially supported by NSFC grant 12271284 and EIMI fund by Saint Petersburg State University} \\ {\small Beijing Institute of Mathematical Sciences and Applications}\and 
Laurent Saloff-Coste\thanks{Partially supported by NSF grant 
 DMS 2054593 and DMS 2343868} \\
{\small Department of Mathematics}\\
{\small Cornell University}  }
\maketitle
\begin{abstract}
This work provides an extension of parts of the classical finite dimensional sub-elliptic theory in the context of infinite dimensional compact connected metrizable groups. Given a well understood and well behaved bi-invariant Laplacian, $\Delta$, and a sub-Laplacian, $L$, to which intrinsic distances, $d_\Delta$, $d_L$, are naturally attached, we show that a comparison inequality of the form $d_L\le C(d_\Delta)^c$ (for some $0<c\le  1$) implies that the Dirichlet form of a fractional power of $\Delta$ is dominated by the Dirichlet form associated with $L$. We use this result to show that, under additional assumptions, certain good properties of the heat kernel for $\Delta$ are then passed to the heat kernel  associated with $L$. Explicit examples on the infinite product of copies of $SU(2)$ are discussed to illustrate these results.
\end{abstract}

\section{Introduction}
Symmetric Gaussian convolution semigroups of probability measures (with full support) on compact groups are families of convolution transition kernels $(\mu_t)_{t>0}$\ for reversible left-invariant (non-degenerate) diffusions on these groups. They yield heat semigroups via $$H_tf(x)=f*\mu_t(x).$$ Their infinitesimal generators $-L$\ are left-invariant sub-Laplacians (up to minus sign), formally $H_t=\exp{(-tL)}$, see (\ref{laplacian}). Unlike in the finite dimensional or typical infinite dimensional Hilbert space scenarios, Gaussian semigroups on compact groups may have a variety of different analytic properties \cite{Bendikovbook,centralJFA,survey}. In particular, they may or may not admit a density function with respect to the Haar measure; when the density function exists it is often called the heat kernel. Among all such semigroups, those that are central and  those that are of product type have been investigated most closely and are relatively understood, see \cite{Bendikovbook,Elliptic,centralsurvey} and the references therein. Central Gaussian semigroups correspond to bi-invariant diffusions; they share many properties with product type Gaussian semigroups due to the structure of compact groups \cite{compactgroup,centralJFA}. In contrast, little is known so far about general left-invariant Gaussian semigroups on compact groups. This paper is a continuation of  \cite{infdim} and  attempts to study (some) noncentral Gaussian semigroups via the approach of Dirichlet form perturbation and comparison. In \cite{infdim}, left-invariant Laplacians with Dirichlet forms comparable to the Dirichlet form of some bi-invariant Laplacian are considered, they are associated with {\em elliptic} diffusions. In the present work, comparisons of Dirichlet forms up to fractional powers are developed. This allows us to cover certain sub-elliptic diffusions whose generators are sub-Laplacians satisfying the H\"{o}rmander condition, in some sense. One common feature of the results in both works is that good properties of bi-invariant Laplacians carry over to left-invariant (sub-)Laplacians having comparable Dirichlet forms (in some sense, possibly involving fractional powers). Our choice of comparison conditions in the present paper originates in the classical sub-elliptic theory and involves fractional powers of the bi-invariant Laplacian, see e.g. \cite{Hormander,RothschildStein,FeffermanPhong,subelliptic,MustaphaVaropoulos} and the references therein.

In addition, we investigate natural conditions on (sub-)Laplacians that guarantee such Dirichlet form comparison relations. Inspired by well-known equivalence relations in the compact Lie group setting between comparison of Dirichlet forms and comparison of intrinsic distances, both up to fractional powers, we examine in the context of infinite dimensional compact groups the implication from the latter condition to the former (from distances to Dirichlet forms). Due to the complicated nature of volume in infinite dimension, we avoid using  the Besov norm $\Lambda^{2,2}_\alpha$ which is favored in the finite dimension literature. Instead of equivalence, our result is a one-sided implication that a comparison between instrinsic distances, together with some heat kernel bound assumptions on the better understood Gaussian semigroup, implies a comparison between Dirichlet forms, up to fractional powers. This part can be viewed as a partial generalization of the classical sub-elliptic theory to the infinite dimensional compact group setting.

More precisely, let $(G,\nu)$\ be a compact connected metrizable group with normalized Haar measure $\nu$ and identity element $e$. Let $\Delta$\ and $L$\ be two (sub-)Laplacians on $G$, with associated respective intrinsic distances $d_\Delta$\ and $d_L$, and Dirichlet forms $\mathcal{E}_\Delta$\ and $\mathcal{E}_L$.
It is known that these intrinsic distances are finite on a dense subset of $G$ but they can take the value $+\infty$. Consider the property (CK$\lambda$), $\lambda>0$. This is a heat kernel upper bound condition which requires that the given Gaussian semigroup $(\mu_t)_{t>0}$\ admits a continuous density, $\mu_t$, for all $t>0$, satisfying
\begin{eqnarray*}
\sup_{0<t<1} t^{\lambda}\log{\mu_t(e)}<\infty.
\end{eqnarray*}
See Definition \ref{CKlambdadef} below. In this paper, we will mostly focus on this property of the heat kernel. Note that this condition is equivalent to the existence of a constant $C\ge 0$ such that $\mu_t(e)\leq \exp{(Ct^{-\lambda})}$ for $0<t<1$. For small $t$, the smaller $\lambda$\ is, the smaller the upper bound $\exp{(Ct^{-\lambda})}$\ is, and so the stronger the condition (CK$\lambda$) is. In fact it is important to make a difference between $\lambda<1$\ and $\lambda\ge 1$, we will come back to this later. Suppose the Gaussian semigroup $(\mu_t^\Delta)_{t>0}$\ associated with $\Delta$\ satisfies (CK$\lambda$) for some $0<\lambda<1$. Under this assumption, the intrinsic distance $d_\Delta$ is finite and continuous on $G$ (it also defines the natural topology of $G$). Now, if we assume
that the distances $d_\Delta$\ and $d_L$\ are related by
\begin{eqnarray*}
d_L\leq C(d_\Delta)^c\ \mbox{for some }0<c<1,\ C>0,
\end{eqnarray*}
where $\lambda$\ and $c$\ satisfy that $(1-\lambda)c-2\lambda>0$, then we will prove that, for any $0<\epsilon<(1-\lambda)c-2\lambda$,
\begin{eqnarray}
\label{formcompare1}
\mathcal{E}_{\Delta^\epsilon}(f,f)=\|\Delta^{\epsilon/2} f\|_{L^2(G)}^2\leq C(\epsilon) (\mathcal{E}_L(f,f)+\|f\|_{L^2(G)}^2)
\end{eqnarray}
for some constant $C(\epsilon)>0$\ and any function $f$\ in the domain of $\mathcal{E}_L$, the Dirichlet form of $L$. The notation $\mathcal{E}_{\Delta^\epsilon}$\ denotes the Dirichlet form associated with $\Delta^\epsilon$. This is the content of one of our main results, Theorem \ref{formthm} in Section \ref{comparesection}.

Starting from this comparison result we obtain some estimates of the heat kernel $\mu_t^L$\ for the Gaussian semigroup associated to $L$, namely when the parameters $\lambda$\ and $c$\ satisfy some relation (roughly $0<\lambda<c/3$\ when $c$\ is small), the  Gaussian semigroup associated to $L$ satisfies (CK$\gamma$) for some $\gamma$\ depending on the fractional power $\epsilon$\ in (\ref{formcompare1}); if furthermore $\lambda$\ and $c$\ satisfy roughly the condition $0<\lambda<c/4$, then the new heat kernel $\mu_t^L$\ satisfies (CK$\gamma$) with some $0<\gamma<1$\ as well as  the Gaussian type upper bound 
\begin{eqnarray*}
\mu_t^L(x)\leq \exp{\left\{\frac{C}{t^\gamma}-\frac{C'(d_L(x,e))^2}{t}\right\}},\ \forall 0<t<1,\ \forall x\in G,
\end{eqnarray*}
for some constants $C,C'>0$.
This is the main theorem of Section \ref{Gaussiansection}, Theorem \ref{gaussianestimatethm}. Some corollaries are given in the same section.

Finally in Section \ref{examplesection} we describe concrete examples of $\Delta$ and $L$ on $G=\prod_{1}^\infty SU(2)$\ that satisfy such distance comparison and Dirichlet form comparison relations.

\section{Background}
For general perspectives on convolutional semigroups of probability measures on groups, see \cite{Heyer}. The underlying space considered in this paper is an arbitrary compact connected metrizable group, denoted by $G$, with normalized Haar measure $\nu$\ and identity element $e$. In the following we write compact group for short, and write $L^p(G)$\ for $L^p(G,\nu)$, $\|\cdot\|_p$\ for $\|\cdot\|_{L^p(G)}$, $p\in [1,\infty]$. Such a group $G$\ admits a projective structure as the projective limit of a sequence of compact Lie groups  (see \cite{groupstructure2,compactgroup,groupstructure})
 \begin{eqnarray*}
 G=\varprojlim_{\alpha\in \aleph} G_\alpha.
 \end{eqnarray*}
Here the index set $\aleph$\ is finite or countable, and ordered; $G_\alpha=G/K_\alpha$\ where $\{K_\alpha\}$\ is a decreasing sequence of compact normal subgroups with $\bigcap_{\alpha\in \aleph}K_\alpha=\{e\}$. More precisely, consider the projection maps $\pi_\alpha: G\rightarrow G_\alpha=G/K_\alpha$\ for $\alpha\in\aleph$, and $\pi_{\alpha,\beta}:G_\beta\rightarrow G_\alpha$\ for $\alpha\leq \beta$, $\alpha,\beta\in\aleph$. The group $G$\ is the projective limit of the projective system $(G_\alpha,\pi_{\alpha,\beta})_{\alpha\leq \beta}$.

In this context, we are interested in the concepts of  Gaussian semigroup, sub-Laplacian, and Dirichlet form. See \cite{locallycompact} and the references therein for more detailed discussions about these objects in the larger context of locally compact groups. A \textit{symmetric Gaussian convolution semigroup} $(\mu_t)_{t>0}$\ on $G$\ is a family of probability measures satisfying:
\begin{itemize}
    \item[(i)] (semigroup property) $\mu_t*\mu_s=\mu_{t+s}$, for any $t,s>0$;
    \item[(ii)] (weakly continuous) $\mu_t\rightarrow \delta_e$\ weakly as $t\rightarrow 0$;
    \item[(iii)] (Gaussian) $t^{-1}\mu_t(V^c)\rightarrow 0$\ as $t\rightarrow 0$\ for any neighborhood $V$\ of the identity $e\in G$;
    \item[(iv)] (symmetric) $\check{\mu}_t=\mu_t$\ for any $t>0$, here $\check{\mu}_t$\ is defined by $\check{\mu}_t(V)=\mu_t(V^{-1})$\ for any Borel subset $V\subset G$.
\end{itemize}
For convenience, recall the following convolution formulas.
\begin{itemize}
    \item The convolution of any two Borel measures $\mu_1,\mu_2$\ on $G$, $\mu_1*\mu_2$, is a measure defined by
\begin{eqnarray*}
\mu_1*\mu_2(f)=\int_{G\times G}f(xy)\,d\mu_1(x)d\mu_2(y), \ \forall f\in C(G).
\end{eqnarray*}
\item The convolution of a function $f\in C(G)$\ and a Borel measure $\mu$\ is a function defined as
\begin{eqnarray*}
f*\mu(x)=\int_G f(xy^{-1})d\mu(y),\ \mu*f(x)=\int_G f(y^{-1}x)\,d\mu(y).
\end{eqnarray*}
\item The convolution of any two functions $f,g\in C(G)$\ is the function given by
\begin{eqnarray*}
f*g(x)=\int_Gf(xy^{-1})g(y)\,d\nu(y)=\int_Gf(y)g(y^{-1}x)\,d\nu(y).
\end{eqnarray*}
\end{itemize}
For $t>0$, set
\begin{eqnarray*}
H_tf(x):=f*\check{\mu}_t(x)=\int_Gf(xy)\,d\mu_t(y)=f*\mu_t(x),\ \forall f\in C(G),
\end{eqnarray*}
then extend it to $L^2(G)$. This defines a Markov semigroup of self-adjoint operators on $L^2(G)$, which we refer to as the \textit{heat semigroup}.  Let $-L$\ be the ($L^2$-)infinitesimal generator of $(H_t)_{t>0}$
\begin{eqnarray*}
-Lf:=\lim_{t\rightarrow 0}\frac{H_tf-f}{t},
\end{eqnarray*}
with domain $\mathcal{D}(L)$\ equal to the space of functions in $L^2(G)$\ for which this limit exists in $L^2(G)$. In particular, $L$\ is self-adjoint and non-negative definite. Let $(\mathcal{E}_L,\mathcal{D}(\mathcal{E}_L))$\ be the associated Dirichlet form with domain $\mathcal{D}(\mathcal{E}_L)$ equal to the domain of $L^{1/2}$, where $L^{1/2}$\ is defined via spectral theory. Then
\begin{eqnarray*}
\mathcal{E}_L(f,g)=\langle Lf,g\rangle_{L^2}=\langle L^{1/2}f,L^{1/2}g\rangle_{L^2},\ \forall f\in \mathcal{D}(L),\ \forall g\in \mathcal{D}(\mathcal{E}_L).
\end{eqnarray*}
Here $\langle f,g\rangle_{L^2}:=\int_G fg\,d\nu$\ denotes the $L^2$\ inner product. The one-to-one correspondences among the Markov semigroup, the generator, and the Dirichlet form are standard, see e.g. \cite{Fukushima}.

All such Dirichlet forms share a common core, the space of Bruhat test functions $\mathcal{B}(G)$. It is defined as
\begin{eqnarray*}
\mathcal{B}(G):=\left\{f:G\rightarrow \mathbb{R}:f=\phi\circ\pi_\alpha\ \mbox{for some }\alpha\in \aleph,\ \phi\in C^\infty(G_\alpha)\right\}.
\end{eqnarray*}
Here $C^\infty(G_\alpha)$\ denotes the set of all smooth functions on $G_\alpha$; $\mathcal{B}(G)$\ is independent of the choice of $\{K_\alpha\}_{\alpha\in \aleph}$. The Bruhat test functions are generalizations of cylindric functions on the infinite dimensional torus $\mathbb{T}^\infty$\ (i.e. smooth functions that depend only on finitely many coordinates).

When applied to  Bruhat test functions, generators of symmetric Gaussian semigroups admit an explicit expression which we now recall. Let $\mathfrak{g}$\ be the (projective) Lie algebra of $G$\ defined as the projective limit of the Lie algebras of $G_\alpha$\ (denoted by $\mathfrak{g}_\alpha$), with projection maps $d\pi_{\alpha,\beta}$\ for $\alpha,\beta\in \aleph$, $\alpha\leq \beta$. A family $\{X_i\}_{i\in \mathcal{I}}$\ of elements of $\mathfrak{g}$\ is called a \textit{projective family} if for any $\alpha\in \aleph$, there is a finite subset $\mathcal I_\alpha$ of $\mathcal I$ such that $d\pi_\alpha(X_i)=0$ for all $i\not\in \mathcal I_\alpha$. It is called a \textit{projective basis} of $\mathfrak{g}$, if for each $\alpha\in \aleph$, there is a finite subset $\mathcal{I}_\alpha\subset \mathcal{I}$, such that $d\pi_\alpha(X_i)=0$\ for all $i\notin \mathcal{I}_\alpha$, and $\{d\pi_\alpha(X_i)\}_{i\in \mathcal{I}_\alpha}$\ is a basis of the Lie algebra $\mathfrak{g}_\alpha$. The (projective) Lie algebra $\mathfrak{g}$\ admits many projective bases.

For a fixed projective basis $X=\{X_i\}_{i\in \mathcal{I}}$, generators of symmetric Gaussian semigroups are in one-to-one correspondence with symmetric non-negative real matrices $A$, meaning $A=(a_{ij})_{\mathcal{I}\times \mathcal{I}}$\ where $a_{ij}=a_{ji}$\ are real numbers, and $\sum_{i,j\in\mathcal{I}} a_{ij}\xi_i\xi_j\geq 0$\ for all $\xi=(\xi_i)_{i\in \mathcal{I}}$\ with finitely many nonzero $\xi_i\in\mathbb{R}$. The matrix $A$\ can be degenerate. Associate to each such matrix $A$ the second-order left-invariant differential operator
\begin{eqnarray}
\label{laplacian}
L_A:=-\sum_{i,j\in \mathcal{I}}a_{ij}X_iX_j,
\end{eqnarray}
it acts on $\mathcal{B}(G)$\ by
\begin{eqnarray*}
L_Af(x) = -\sum_{i,j\in \mathcal{I}}a_{ij}(d\pi_\alpha(X_i)d\pi_\alpha(X_j)\phi)(\pi_\alpha(x)),\ \mbox{for }f=\phi\circ \pi_\alpha\in \mathcal{B}(G).
\end{eqnarray*}
The operator $L_A$\  is called a \textit{Laplacian} if $A$\ is positive definite. For any symmetric Gaussian semigroup $(\mu_t)_{t>0}$\ with corresponding heat semigroup $(H_t)_{t>0}$, there is a symmetric non-negative  real matrix $A$\ such that
\begin{eqnarray*}
\lim_{t\rightarrow 0}\frac{H_tf-f}{t}=-L_Af=\sum_{i,j\in \mathcal{I}} a_{ij}X_iX_jf, \ \forall f\in\mathcal{B}(G).
\end{eqnarray*}
Conversely, given an operator $L=L_A$ as above, there is a Markov semigroup that we denote by $(H^L_t)_{t>0}$, whose generator agrees with $-L$\ in the above sense. It corresponds to a Gaussian semigroup $(\mu_t^L)_{t>0}$\ and a Dirichlet form $(\mathcal{E}_L,\mathcal{D}(\mathcal{E}_L))$\ as before.

Any such $L$\ can be written as a sum of squares
\begin{eqnarray*}
L=-\sum_{i\in \mathcal{J}} Y_i^2
\end{eqnarray*}
for some projective basis $\{Y_i\}_{i\in \mathcal{I}}$ and some sub-index set $\mathcal{J}\subseteq \mathcal{I}$. Such choice of projective basis is not unique. The operator $L$ is a Laplacian if and only if  $\mathcal J=\mathcal I$. By definition, it is a {\em sub-Laplacian} if the family $\{Y_i: i\in \mathcal J\}$ generates algebraically (i.e., using Lie brackets) the (projective) Lie algebra of $G$ (i.e., the projection of $\{Y_i: i\in \mathcal J\}$ on each $G_\alpha$ satisfies the H\"ormander condition).

One other quantity associated with such a (strongly local left-invariant) Dirichlet form $\mathcal{E}$, or equivalently with  $L$, is the intrinsic distance, denoted by $d_L$ (it can take the value $+\infty$). It is defined as
\begin{eqnarray}
\label{distdef}
d_L(x,y)=\sup{\left\{f(x)-f(y):f\in \mathcal{B}(G),\ \Gamma_L(f,f)\leq 1\ \mbox{on }G\right\}},
\end{eqnarray}
where for $f,g\in \mathcal{B}(G)$, $\Gamma_L(f,g):=-(L(fg)-(Lf)g-f(Lg))/2$. If $L=-\sum a_{ij}X_iX_j$, then
$\Gamma_L(f,g)=\sum a_{ij}X_ifX_jg$. Note that $\Gamma_L(f,f)$\ is the square of the length of the ``gradient'' of $f$, which can be seen from the expression $\Gamma_L(f,f)=\sum |Y_if|^2$\ if we write $L=-\sum Y_i^2$\ in a sum of squares form. We may equally define the intrinsic distance notion by taking supremum over functions $f$\ that locally belong to $\mathcal{D}(\mathcal{E})$. See e.g. \cite[Section 4]{locallycompact}. When the point $y$ is the identity $e$, we write $d(x):=d(x,e)$.

When $G$\ is infinite dimensional, it is known that Gaussian semigroups with various analytic properties exist on $G$, see e.g. \cite{Elliptic,centralsurvey}. One particular type of property we consider in the present paper is the (CK$\lambda$) property, where $\lambda$\ is a positive real number.
\begin{definition}
\label{CKlambdadef}
A Gaussian semigroup $(\mu_t)_{t>0}$\ is said to satisfy (CK$\lambda$), if
\begin{itemize}
\item[(i)] for all $t>0$, $\mu_t$\ is absolutely continuous with respect to the Haar measure of $G$, and admits a continuous density;
\item[(ii)] furthermore, the continuous density function, denoted again by $\mu_t(\cdot)$, satisfies
\begin{eqnarray*}
\sup_{0<t<1}t^\lambda \log{\mu_t(e)}<\infty.
\end{eqnarray*}
\end{itemize}
\end{definition}
Note that (CK$\lambda$) is an ultracontractivity condition that the corresponding heat semigroup $(H_t)_{t>0}$\ satisfies (recall that $H_t f=f*\mu_t$)
\begin{eqnarray*}
\|H_t\|_{L^1(G)\rightarrow L^\infty(G)}\leq e^{C/t^\lambda},\ 0<t<1,
\end{eqnarray*}
where $C>0$\ is a constant. In the following we write $\|\cdot\|_{1\rightarrow\infty}$\ for $\|\cdot\|_{L^1(G)\rightarrow L^\infty(G)}$.

If $(\mu_t)_{t>0}$\ satisfies (CK$\lambda$), in particular it admits a density function $\mu_t(\cdot)$\ for all $t>0$, and $H_tf(x)=f*\mu_t(x)=\int f(xy)\mu_t(y)d\nu(y)$. With a little abuse of notation we call $\mu_t(\cdot)$\ the heat kernel of the heat semigroup $H_t$, because the heat kernel $h(t,x,y)$\ that satisfies $H_t f(x)=\int f(y)h(t,x,y)d\nu(y)$\ is just $h(t,x,y)=\mu_t(x^{-1}y)$.

The special value $\lambda=1$\ turns out to be very important. For instance, if $0<\lambda<1$, 
then the associated intrinsic distance $d$ is finite and continuous and that the density function $\mu_t$\ satisfies a full Gaussian type estimate
\begin{eqnarray*}
\mu_t(x)\leq \exp{\left\{\frac{C}{t^\gamma}-\frac{C'd^2(x)}{t}\right\}},\ \forall 0<t<1,\ \forall x\in G,
\end{eqnarray*}
where $C,C'>0$, and recall that $d(x):=d(x,e)$.  When $\lambda\geq 1$, the condition (CK$\lambda$) is much weaker (it does not imply that $d$ is finite), and there is no such Gaussian type bound of interest. See \cite{orange} for details.

\section{Comparisons between distances and between Dirichlet forms}
\label{comparesection}
Let $G$\ be a compact group and $\Delta$\ be a sub-Laplacian on $G$, in particular, $-\Delta$\ is the generator of some heat semigroup $(H_t^\Delta)_{t>0}$. We write $H_t$\ for $H_t^\Delta$\ in this section when there is no ambiguity. By spectral theory, for any $0<\epsilon\leq 1$,
\begin{eqnarray}
\label{tintegral}
\mathcal{E}_{\Delta^{\epsilon}}(f,f)=\|\Delta^{\epsilon/2} f\|^2_2=C\int_0^\infty \left(\frac{t\|\partial_tH_tf\|_2}{t^{\epsilon/2}}\right)^2\frac{dt}{t}
\end{eqnarray}
for some constant $C=C(\epsilon)>0$. In the classical setting of $\mathbb{R}^n$, this norm defines the homogeneous $L^2$\ Sobolev norm of order $\epsilon$, which is also equivalent 
to the $\Lambda^{2,2}_\epsilon$\ Besov norm. Divide the integral into two parts, one from $0$\ to $1$\ and the other from $1$\ to $\infty$. Denote $\partial_t:=\partial/\partial t$. Because $\|\partial_tH_tf\|_2\leq (et)^{-1}\|f\|_2$, the second part satisfies for some $C'>0$,
$$\int_{1}^{\infty}\left(\frac{t\|\partial_tH_tf\|_2}{t^{\epsilon/2}}\right)^2\frac{dt}{t}\leq e^{-2}\int_{1}^{\infty}\frac{\|f\|_2^2}{t^{1+\epsilon}}\,dt\leq C'\|f\|_2^2.$$
We are going to show that, under additional conditions, the part of the integral from $0$\ to $1$\ is controlled by the Dirichlet form $\mathcal{E}_L$\ of $L$. We state this estimate in the next theorem.

Another norm that is relevant to us and will appear in the proof of that theorem is the $\Lambda^{2,\infty}_{\alpha}$\ Besov-Lipschitz norm defined in terms of the intrinsic distance. Given any (sub-)Laplacian $P$\ and its corresponding intrinsic distance $d_P$,  for any $0<\alpha<1$, define the inhomogeneous $\Lambda^{2,\infty}_{\alpha,P}$\ Besov-Lipschitz norm as
\begin{eqnarray*}
\Lambda^{2,\infty}_{\alpha,P}(f):=\|f\|_2+\sup_{\substack{y\in G\\y\neq e}}\frac{\|f(\cdot y)-f(\cdot)\|_2}{d_P(y)^\alpha},\ \forall f\in L^2(G).
\end{eqnarray*}
In general, for the definition and discussion of the $\Lambda^{p,q}_\alpha$\ norms in $\mathbb{R}^n$\ for $1\leq p,q\leq \infty$, $\alpha>0$, see for example \cite[Chapter V Section 5]{Steinbook}. For $q\neq \infty$, the generalization of this norm to our setting is problematic.

We now state the theorem, which says that under some additional assumptions, the Sobolev norm of fractional order (associated with $\Delta$) is controlled by the Dirichlet form of $L$.
\begin{theorem}
\label{formthm}
Let $G$\ be a compact connected metrizable group. Let $\Delta$\ be a sub-Laplacian on $G$\ with associated Gaussian semigroup $(\mu_t^\Delta)_{t>0}$\ that satisfies (CK$\lambda$), $0<\lambda<1$. Let $L$\ be another sub-Laplacian on $G$. Suppose the associated intrinsic distances $d_\Delta$\ and $d_L$\ satisfy the comparison relation
\begin{eqnarray}
\label{distcompare}
d_L\leq C(d_\Delta)^c
\end{eqnarray}
where $0<c\leq 1$, $C>0$. Assume further that
\begin{eqnarray*}
(1-\lambda)c-2\lambda>0.
\end{eqnarray*}
Then the following two norm comparisons hold.\\
(1) For any $2\lambda[(1-\lambda)c]^{-1}<\alpha< 1$ and any 
\begin{eqnarray*}
0<\epsilon<(1-\lambda)\alpha c-2\lambda,
\end{eqnarray*}
\begin{eqnarray} \label{comp1}
\|\Delta^{\epsilon/2} f\|_2\leq C\Lambda^{2,\infty}_{\alpha,\,L}(f).
\end{eqnarray}
(2) For any
\begin{eqnarray*}
0<\epsilon<(1-\lambda)c-2\lambda,
\end{eqnarray*}
\begin{eqnarray}
\label{formcomparison}
\|\Delta^{\epsilon/2} f\|_2^2\leq C(\mathcal{E}_L(f,f)+\|f\|_2^2).
\end{eqnarray}
\end{theorem}
Before giving the proof of (\ref{comp1}), we observe that (\ref{formcomparison})
follows from (\ref{comp1}) thanks to the following lemma.
\begin{lemma} \label{lemP}
Let $G$\ be a compact connected metrizable group. Let $L$\ be a sub-Laplacian on $G$ such that 
$\sup_{x,y\in G}\{d_L(x,y)\}=\mbox{\em diam}_L(G)<+\infty$. Then for any $\beta\in [0,1]$ and $y\in G$,
\begin{equation} \label{Besovcomp}
\sup_{\substack{y\in G\\y\neq e}}\frac{\|f(\cdot y)-f(\cdot)\|_2}{d_L(y)^{\beta}}\leq \max\{1,\mbox{\em diam}_L(G)^{1-\beta}\}\left(\mathcal{E}_L(f,f)\right)^{1/2}.
\end{equation}
\end{lemma}
\begin{proof}
 This is a simple consequence of a basic Poincar\'e-type inequality. More precisely, suppose $L=-\sum_{i\in\mathcal{I}} Y_i^2$\ for a projective family $\{Y_i\}_{i\in\mathcal{I}}$. Because $L$ is a sub-Laplacian, for any $\alpha\in\aleph$, the set $\{d\pi_\alpha(Y_i)\}_{i\in\mathcal{I}}$\ satisfies the H\"{o}rmander condition on $G_\alpha=G/K_\alpha$. For any Bruhat test function $f\in \mathcal{B}(G)$, suppose $f=\phi\circ\pi_l$\ where $l\in \aleph$, $\phi$\ is smooth on $G_l=G/K_l$. On the Lie group $G_l$, let $d_{L,l}$\ denote the intrinsic distance associated with $d\pi_l (L)=-\sum (d\pi_l(Y_i))^2$, and denote the identity element as $e_l$. In our notation, $d_{L,l}(\pi_l(y))=d_{L,l}(\pi_l(y),e_l)$\ measures the (sub-elliptic) distance between $e_l$\ and $\pi_l(y)$, and there is some curve $\gamma_l:[0,d_{L,l}(\pi_l(y))]\rightarrow G_l$\ with unit (sub-elliptic) speed, such that $\gamma_l(0)=e_l$, $\gamma_l(d_{L,l}(\pi_l(y)))=\pi_l(y)$. See \cite[Sections 5.5.3, 5.5.4]{SturmGeom}. Then
\begin{eqnarray}
\lefteqn{|f(xy)-f(x)|=|\phi(\pi_l(x)\pi_l(y))-\phi(\pi_l(x))|}\notag\\
&\leq& \int_{0}^{d_{L,l}(\pi_l(y))}\left(\sum_{i\in\mathcal{I}}\left|d\pi_l(Y_i)\phi(\pi_l(x)\gamma_l(s))\right|^2\right)^{1/2}\,ds. \label{poincareprojection}
\end{eqnarray} 
Note that $Y_if=(d\pi_l(Y_i)\phi)\circ \pi_l$; the summation in (\ref{poincareprojection}) is a finite summation because $\{Y_i\}_{i\in\mathcal{I}}$\ is a projective family. Integrating the square of both sides in (\ref{poincareprojection}) over $x\in G$\ gives
\begin{eqnarray*}
\lefteqn{\int_G |f(xy)-f(x)|^2\,d\nu(x)}\\
&\leq& d_{L,l}(\pi_l(y))\int_{0}^{d_{L,l}(\pi_l(y))}\int_G\sum_{i\in\mathcal{I}}\left|d\pi_l(Y_i)\phi(\pi_l(x)\gamma_l(s))\right|^2\,d\nu(x)ds\\
&=& d_{L,l}(\pi_l(y))\int_{0}^{d_{L,l}(\pi_l(y))}\int_G\sum_{i\in\mathcal{I}}\left|Y_if(x)\right|^2\,d\nu(x)ds.
\end{eqnarray*}
Jensen's inequality and the Fubini theorem are used to obtain the first inequality, and right-invariance of the Haar measure is applied in the last line. Because (see \cite[Theorem 4.2]{locallycompact})
\begin{eqnarray*}
d_L(y)=d_L(e,y)=\sup_{\alpha\in\aleph}d_{L,\alpha}(e_\alpha,\pi_\alpha(y)),
\end{eqnarray*}
it follows that
\begin{eqnarray*}
\int_G |f(xy)-f(x)|^2\,d\nu(x) &\leq &d_L(y)^2\int_G\sum_{i\in\mathcal{I}}\left|Y_if(x)\right|^2\,d\nu(x)\\
&=& d_L(y)^2\mathcal E_L(f,f).\end{eqnarray*}
Dividing (the square roots of) both sides by $d_L(y)^{\beta}$ and using $\beta\in [0,1]$ and the assumption that $\mbox{diam}_L(G)<+\infty$, we obtain (\ref{Besovcomp}).
\end{proof}

\begin{proof}[Proof of (\ref{comp1})]
By Lemma \ref{lemP} and the discussion above, it remains to estimate the integral part from $0$\ to $1$\ in (\ref{tintegral}). In this proof we write $H_t$\ for $H_t^\Delta$\ and $\mu_t$\ for $\mu_t^\Delta$. Because $\int_G \partial_t\mu_t(y)\,d\nu(y)=0$, 
\begin{eqnarray*}
\|\partial_tH_tf\|_2\leq \int_G \|f(\cdot y)-f(\cdot)\|_2\,|\partial_t\mu_t(y)|\,d\nu(y).
\end{eqnarray*}
Then for any $\beta\geq 0$,
\begin{eqnarray*}
\|\partial_tH_tf\|_2
\leq \sup_{\substack{y\in G\\y\neq e}}\frac{\|f(\cdot y)-f(\cdot)\|_2}{d_\Delta(y)^\beta}\,\int_G d_\Delta(y)^\beta|\partial_t\mu_t(y)|\,d\nu(y).
\end{eqnarray*}
Note that the property (CK$\lambda$) with $0<\lambda<1$\ guarantees that $d_\Delta$\ is a continuous function \cite[Theorem 3.2]{orange} and moreover defines the topology of $G$\ \cite[Theorem 4.2]{locallycompact}. By the comparison relation (\ref{distcompare}) between $d_\Delta$\ and $d_L$, $d_L$\ is a continuous function too, and for some $C>0$,
\begin{eqnarray*}
\sup_{\substack{y\in G\\y\neq e}}\frac{\|f(\cdot y)-f(\cdot)\|_2}{d_\Delta(y)^\beta}\leq C\sup_{\substack{y\in G\\y\neq e}}\frac{\|f(\cdot y)-f(\cdot)\|_2}{d_L(y)^{\beta/c}}.
\end{eqnarray*}
For constants that we do not track the exact values or expressions of, we often denote them all by $C,C'$. Thus (assuming $0\leq \beta<c$),
\begin{eqnarray*}
\|\partial_tH_tf\|_2 &\leq& C\sup_{\substack{y\in G\\y\neq e}}\frac{\|f(\cdot y)-f(\cdot)\|_2}{d_L(y)^{\beta/c}}\,\int_G d_\Delta(y)^\beta|\partial_t\mu_t(y)|\,d\nu(y)\\
&\le & C\Lambda^{2,\infty}_{\beta/c,\,L}(f) \,\int_G d_\Delta(y)^\beta|\partial_t\mu_t(y)|\,d\nu(y).\end{eqnarray*}

To estimate the integral factor, we need a key result, \cite[Theorem 4.1]{Analytic} (see also \cite{der}), which asserts that 
\[
t\int_G\left|\partial_t\mu_t(x)\right| d\nu(x)\le 2e \max\{M_0(t/2), 2\}, \quad \forall t>0, 
\]
where $M_0(t)=\log \mu_t(e)$. Here, by assumption, $M_0(t)\le Ct^{-\lambda}$ because $\mu_t=\mu^\Delta_t$\ satisfies (CK$\lambda$). This also implies that for some $C_1,C_2>0$,
\begin{eqnarray*}
|\partial_t\mu_t(x)|\leq \exp{\left\{\left(\frac{C_1}{t^\lambda}-\frac{C_2d_\Delta(x)^2}{t}\right)\right\}},\ \forall 0<t<1,
\end{eqnarray*}
see e.g. \cite[Theorem 4.2]{centralderivative}. Pick some $R$\ such that $C_1/t^\lambda\leq C_2R^2/2t$, for example $R=(2C_1/C_2)^{1/2}t^{(1-\lambda)/2}=Ct^{(1-\lambda)/2}$. For $t$\ small, $R$\ does not exceed the diameter of $G$. Write
\begin{eqnarray*}
\lefteqn{\int_G d_\Delta(y)^\beta|\partial_t\mu_t(y)|\,d\nu(y)}\\
&=&\int_{d_\Delta(y)\leq R}d_\Delta(y)^\beta|\partial_t\mu_t(y)|\,d\nu(y)+\int_{d_\Delta(y)>R}d_\Delta(y)^\beta|\partial_t\mu_t(y)|\,d\nu(y)\\
&\leq& R^\beta \frac{2e\max{\{M_0(t/2),2\}}}{t}+\int_{d_\Delta(y)>R}d_\Delta(y)^\beta\exp{\left\{\left(-\frac{C_2d_\Delta(y)^2}{2t}\right)\right\}}\,d\nu(y).
\end{eqnarray*}
For the second integral, the maximum of the integrand over $d_\Delta\in\mathbb{R}_+$\ is obtained at $d_\Delta(y)=\left(\beta t/C_2\right)^{1/2}$, and equals $\left(\beta t/C_2\right)^{\beta/2}e^{-\beta/2}$. Plugging in the upper bounds, we get that
\begin{eqnarray*}
\lefteqn{\int_G d_\Delta(y)^\beta|\partial_t\mu_t(y)|\,d\nu(y)}\\
&\leq& \left(Ct^{\frac{1-\lambda}{2}}\right)^\beta C't^{-\lambda-1}+\left(\frac{\beta t}{C_2}\right)^{\frac{\beta}{2}}e^{-\frac{\beta}{2}}\nu(G)\\
&\leq& C''t^{(1-\lambda)\beta/2-\lambda-1},\ \mbox{for }0<t<1.
\end{eqnarray*}

Hence
\begin{eqnarray*}
\lefteqn{\|\Delta^{\epsilon/2} f\|_2^2\leq C\|f\|_2^2+}\\
&&C'\int_0^1 t^{1-\epsilon} \left(\sup_{\substack{y\in G\\y\neq e}}\frac{\|f(\cdot y)-f(\cdot)\|_2}{(d_L(y))^{\beta/c}}\right)^2
\left(\int_G(d_\Delta(y))^\beta|\partial_t\mu_t(y)|\,d\nu(y)\right)^2\,dt\\
&\leq& C\|f\|_2^2+C\mathcal{E}_L(f,f)\int_0^1 t^{1-\epsilon+(1-\lambda)\beta-2\lambda-2}\,dt. 
\end{eqnarray*}
For the integral in $t$\ to converge, we need
\begin{eqnarray*}
1-\epsilon+(1-\lambda)\beta-2\lambda-2>-1,
\end{eqnarray*}
giving
\begin{eqnarray*}
0<\epsilon<(1-\lambda)\beta-2\lambda.
\end{eqnarray*}
This proves (\ref{comp1}) with $\alpha:=\beta/c$. As mentioned earlier, the comparison relation (\ref{formcomparison}) then follows from (\ref{comp1}) and Lemma \ref{lemP}. To find the range of $\epsilon$, letting $\beta=c$\ (or $\alpha=1$) and we obtain that
\begin{eqnarray*}
0<\epsilon<(1-\lambda)c-2\lambda.
\end{eqnarray*}
\end{proof}

\section{Gaussian type upper bounds for sub-Laplacians}
\label{Gaussiansection}
One consequence of the (CK$\lambda$) assumption is that a Nash type inequality holds for $\Delta$. More precisely, the ultracontractivity property of the heat semigroup $\|H_t^\Delta\|_{1\rightarrow \infty}\leq \exp{\{c/t^\lambda\}}$\ for all $t>0$, is equivalent to the Nash inequality
\begin{eqnarray*}
\|f\|_2^2\left[\log_+{(\|f\|_2^2)}\right]^{1+1/\lambda}\leq C\langle\Delta f,\,f\rangle_{L^2},
\end{eqnarray*}
for $f$\ in the domain $\mathcal{D}(\Delta)$ of $\Delta$, and $\|f\|_1=1$. Note that for the Nash inequality it suffices to verify for Bruhat test functions. By \cite[Theorem 3.1]{Nashfractional}, for $\epsilon>\lambda/(1+\lambda)$, the fractional power $\Delta^{\epsilon}$\ satisfies the Nash inequality
\begin{eqnarray}
\|f\|_2^2\left[\log_+{(\|f\|_2^2)}\right]^{\epsilon(1+1/\lambda)}\leq C'\langle\Delta^{\epsilon}f,\,f\rangle_{L^2}\label{Nashineq}
\end{eqnarray}
for $f\in\mathcal{D}(\Delta^{\epsilon})$, $\|f\|_1=1$, which then implies that for $t>0$,
\begin{eqnarray*}
\|e^{-t\Delta^{\epsilon}}\|_{1\rightarrow\infty}\leq e^{c't^{-\beta}},\ \mbox{where } \beta=\frac{\lambda}{\epsilon(1+\lambda)-\lambda}>0.
\end{eqnarray*}
Requiring $\epsilon>\lambda/(1+\lambda)$\ is to guarantee that $\beta>0$. When Theorem \ref{formthm} holds, it then follows that $L$\ satisfies the local Nash inequality
\begin{eqnarray}
\label{NashforL}
\|f\|_2^2[\log_+{(\|f\|_2^2)}]^{\epsilon(1+1/\lambda)}\leq C\|f\|_2^2+C'\mathcal{E}_L(f,f).
\end{eqnarray}
For this inequality to imply the ultracontractivity of $H_t^L$, the exponent $\epsilon(1+1/\lambda)$\ needs to be larger than $1$, which again is ensured by $\epsilon>\lambda/(1+\lambda)$. Hence we have the following theorem.
\begin{theorem}
\label{gaussianestimatethm}
Let $G$\ be a compact connected metrizable group. Let $\Delta$\ and $L$ be two sub-Laplacians on $G$. Suppose the Gaussian semigroup $(\mu_t^\Delta)_{t>0}$\ associated with $\Delta$\ satisfies (CK$\lambda$), where $0<\lambda<1$. Suppose the intrinsic distances associated with $\Delta$\ and $L$\ satisfy (\ref{distcompare}), i.e.
\begin{eqnarray*}
d_L\leq C(d_\Delta)^c
\end{eqnarray*}
for some $0<c\leq 1$, $C>0$. Suppose $\lambda$\ and $c$\ satisfy
\begin{eqnarray}
\label{boundforlambda}
0<\lambda<\frac{-3+\sqrt{9+4(2+c)c}}{2(2+c)}\approx \frac{1}{3}c.
\end{eqnarray}
Then the Gaussian semigroup $(\mu_t^L)_{t>0}$\ of the sub-Laplacian $L$\ satisfies (CK$\gamma$) for any $\gamma$\ in the range
\begin{eqnarray*}
\frac{\lambda}{-(2+c)\lambda^2-3\lambda+c}<\gamma<\infty.
\end{eqnarray*}
Set $\gamma_{\lambda,c}:=\lambda\left(-(2+c)\lambda^2-3\lambda+c\right)^{-1}$. If furthermore 
\begin{eqnarray}
\label{boundforlambda1}
0<\lambda<\frac{-2+\sqrt{4+c(2+c)}}{2+c}\approx \frac{c}{4},
\end{eqnarray}
then $\gamma_{\lambda,c}<1$, the heat kernel $\mu_t^L$\ satisfies (CK$\gamma$) for any $\gamma$\ between $\gamma_{\lambda,c}$\ and $1$, and hence satisfies the Gaussian type upper bound ($C,C'>0$\ are constants that depend on $\gamma$)
\begin{eqnarray*}
\mu_t^L(x)\leq \exp{\left\{\frac{C}{t^\gamma}-\frac{C'd_L^2(x)}{t}\right\}},\ 0<t<1.
\end{eqnarray*}
\end{theorem}
\begin{remark}
In (\ref{boundforlambda}) and (\ref{boundforlambda1}), the approximate upper bounds $c/3$\ and $c/4$\ are for when $c$\ is close to $0$. Note that (\ref{boundforlambda}) implies $\gamma_{\lambda,c}>0$.
\end{remark}
\begin{proof}
As discussed above, the conclusions follow from the local Nash inequality (\ref{NashforL}) with admissible $\epsilon$, because by considering $\widetilde{H}_t:=e^{-bt}H_t^L$\ for some proper constant $b>0$, denote its generator by $-\widetilde{L}$, there is some constant $C''$\ such that
\begin{eqnarray}
\|f\|_2^2[\log_+{(\|f\|_2^2)}]^{\epsilon(1+1/\lambda)}\leq C''\mathcal{E}_{\widetilde{L}}(f,f).\label{NashforL1}
\end{eqnarray}
It then follows that $\widetilde{H}_t$\ satisfies an ultracontractivity bound for all $t>0$, hence so does $H_t^L$\ for $0<t<1$. Note also that the continuity of the density function $\mu_t^L(\cdot)$\ follows from its boundedness because of its convolution form. To obtain (\ref{NashforL}) using Theorem \ref{formthm}, we require that
\begin{eqnarray}
\label{epsilonrange}
\frac{\lambda}{1+\lambda}<\epsilon<(1-\lambda)c-2\lambda,
\end{eqnarray}
where the purpose of the first inequality is to obtain that $\Delta^{\epsilon}$\ satisfies a useful Nash inequality (i.e., with $\epsilon>\lambda/(1+\lambda)$), and the second inequality is to control $\mathcal{E}_{\Delta^{\epsilon}}$\ by $\mathcal{E}_L$ (i.e., (\ref{formcomparison})). Solving for the interval in (\ref{epsilonrange}) to be nonempty gives
\begin{eqnarray}
\label{computation}
(2+c)\lambda^2+3\lambda-c<0,
\end{eqnarray}
that is,
\begin{eqnarray*}
0<\lambda<\frac{-3+\sqrt{9+4(2+c)c}}{2(2+c)},
\end{eqnarray*}
which is (\ref{boundforlambda}). For small $c$, the upper bound is approximately $c/3$.

We next compute the ultracontractivity parameter $\gamma$\ for $H_t^L$. Setting
\begin{eqnarray*}
\epsilon\left(1+\frac{1}{\lambda}\right)=1+\frac{1}{\gamma}
\end{eqnarray*}
for the exponent in (\ref{NashforL1}), we get
\begin{eqnarray*}
\gamma=\frac{\lambda}{\epsilon+\epsilon\lambda-\lambda}.
\end{eqnarray*}
Note that for $\epsilon$\ in the range (\ref{epsilonrange}), the denominator is positive and $\gamma$\ is decreasing in $\epsilon$. Plugging in the bounds for $\epsilon$\ we get
\begin{eqnarray*}
\frac{\lambda}{-(2+c)\lambda^2-3\lambda+c}<\gamma<\infty.
\end{eqnarray*}
The lower bound is what we call $\gamma_{\lambda,c}$, which is increasing in $\lambda$. When $\gamma_{\lambda,c}$\ is less than $1$, the heat kernel $\mu_t^L$\ satisfies the full Gaussian type upper bound. Letting
\begin{eqnarray*}
\frac{\lambda}{-(2+c)\lambda^2-3\lambda+c}<1,
\end{eqnarray*}
we obtain (from (\ref{computation}), the denominator is positive)
\begin{eqnarray*}
0<\lambda<\frac{-2+\sqrt{4+c(2+c)}}{2+c},
\end{eqnarray*}
which is (\ref{boundforlambda1}).
For obtaining the full Gaussian type upper bound in this case, see for example \cite[Theorem 3.1]{orange}.
\end{proof}

The property (CK$\gamma$) has many applications to the heat equation associated with $L$. We list a few in the next corollary. See \cite[Section 2.2]{potentialtheory} for the definition of Brelot harmonic sheaf and \cite{Bendikovbook} for background and further references.

\begin{corollary}
Under the general hypotheses of Theorem \ref{gaussianestimatethm}, the heat kernel $\mu_t^L$\ of $L$\ is smooth in time for $0<t<1$, and the time derivatives $\partial_t^k\mu_t^L$, $k\in \mathbb{N}_+=\{1,2,3,\cdots\}$, satisfy
\begin{eqnarray}
\sup_{0<t<1}t^\gamma\log{\left\|\partial_t^k\mu_t^L\right\|_\infty}<\infty,\label{derivativeestimate}
\end{eqnarray}
with $\gamma_{\lambda,c}<\gamma<\infty$.

Under the additional assumption (\ref{boundforlambda1}), the following conclusions hold.\\
(1) The harmonic sheaf associated with the sub-Laplacian $L$\ is Brelot.\\
(2) The following parabolic Harnack principle holds. For any open connected subset $V\subset G$, any interval $I=(a,b)$, any compact sets $K_+,K_-\subset I\times V$\ satisfying $K_+\subset (r,b)\times V$, $K_-\subset (a,r)\times V$\ for some $a<r<b$, there exists a constant $C=C(I,V,K_+,K_-)\geq 1$, such that any non-negative local weak solution $u$\ of $(\partial_t+L)u=0$\ satisfies
\begin{eqnarray*}
\sup_{(t,x)\in K_+}u(t,x)\leq C\inf_{(t,x)\in K_-}u(t,x).
\end{eqnarray*}
(3) The time derivatives of the heat kernel, $\partial_t^k\mu_t^L$, further satisfy
\begin{eqnarray*}
\left|\partial_t^k\mu_t^L(x)\right|\leq \exp{\left\{\frac{C_k}{t^\gamma}-\frac{C_k'd_L^2(x)}{t}\right\}},\ \forall 0<t<1.
\end{eqnarray*}
\end{corollary}

\begin{proof}
Note that $\mu_t^L=H_{t/2}^L\mu_{t/2}^L$\ is a solution of the heat equation $(\partial_t+L)u=0$\ and is smooth in time in the $L^2$\ sense (by spectral theory). By their convolution form $\partial_t^k\mu_t^L=L^k\left(\mu_{t/2}^L*\mu_{t/2}^L\right)=\mu_{t/2}^L*L^k\mu_{t/2}^L$\ we know that $\partial_t^k\mu_t^L$\ are continuous functions, $k\in \mathbb{N}_+$. The on-diagonal estimate for $\partial_t^k\mu_t^L$\ can be proved by observing that
\begin{eqnarray*}
\lefteqn{\|\partial_t^k\mu_t^L\|_{2}^2=\|L^k\mu_t^L\|_2^2=\|L^kH_t^L\|_{2\rightarrow\infty}^2}\\
&\leq& \|L^kH_{t/2}^L\|_{2\rightarrow 2}^2\|H_{t/2}^L\|_{2\rightarrow \infty}^2\leq \left(\frac{2k}{t}\right)^k\mu_t^L(e).
\end{eqnarray*}
See for example the proof of \cite[Lemma 2.2]{centralderivative}. The rest is shown by plugging in the (CK$\gamma$) bound for $\mu_t^L(e)$,
\begin{eqnarray*}
\lefteqn{|\partial_t^k\mu_t^L(x)|=|\mu_{t/2}^L*L^k\mu_{t/2}^L(x)|\leq \|\mu_{t/2}^L\|_2\|L^k\mu_{t/2}^L\|_2}\\
&\leq& (\mu_t^L(e))^{1/2}\left(\frac{4k}{t}\right)^{k/2}(\mu_{t/2}^L(e))^{1/2}\leq \exp{\left\{\frac{C_k}{t^\gamma}\right\}},\ \forall 0<t<1.
\end{eqnarray*}
Thus $\partial_t^k\mu_t^L$\ satisfies (\ref{derivativeestimate}).

Next we assume the additional assumption (\ref{boundforlambda1}). Items (1) and (2) are potential theoretic conditions. By \cite[Theorem 4.10]{potentialtheory}, both (1) and (2) follow from the (CK$\sharp$) property of $(\mu_t^L)_{t>0}$, which by definition means that $(\mu_t^L)_{t>0}$\ admits a continuous density function for all $t>0$, denoted again by $\mu_t^L(\cdot)$, and that for any compact set $K\subset G$\ with $e\notin K$,
\begin{eqnarray*}
\lim_{t\rightarrow 0}\sup_{x\in K}\mu_t(x)=0.
\end{eqnarray*}
To show this for $\mu_t^L$, we need \cite[Corollory 3.9]{orange}. More precisely, because the Dirichlet form $\mathcal{E}_L$\ on $G$\ admits the good algebra $\mathcal{B}(G)$\ as required by that corollary, then (CK$\gamma$) of $\mu_t^L$\ with 
$0<\gamma<1$\ (which in particular guarantees the (CKU$*$) property in that corollary) implies (CK$\sharp$) of $\mu_t^L$.
Finally, Item (3) concerns the full Gaussian type upper bound in the case $0<\gamma<1$. As in the last part of the proof of Theorem \ref{gaussianestimatethm}, it follows from the on-diagonal estimate (\ref{derivativeestimate}) by arguing as in \cite[Theorem 3.1]{orange}.
\end{proof}

Lastly in this section we present a corollary regarding a limiting case for the property (CK$\lambda$), namely when the parameter $0<\lambda<1$\ can be taken arbitrarily small. More precisely, we say that a heat semigroup $(H_t)_{t>0}$\ satisfies (CK$0^+$), if it satisfies (CK$\lambda$) for any $0<\lambda<1$, or equivalently, 
\begin{eqnarray*}
\lim_{t\rightarrow 0}e^{-t^{-\lambda}}\|H_t\|_{1\rightarrow \infty}=0, \ \forall 0<\lambda<1.
\end{eqnarray*}
Intuitively, the $L^\infty$\ bound of the heat kernel grows slower than any $\exp{(t^{-\lambda})}$\ function with $0<\lambda<1$, when $t$\ goes to $0$. Applying Theorem \ref{gaussianestimatethm}, we conclude that the property (CK$0^+$) is preserved by the Dirichlet form or distance comparisons considered above.
\begin{corollary}\label{cor-CK0}
Let $G$\ be a compact connected metrizable group. Let $\Delta$\ be a sub-Laplacian satisfying {\em (CK$0^+$)}. Let $L$\ be a sub-Laplacian. If in addition either (\ref{formcomparison})
\begin{eqnarray*}
\mathcal{E}_{\Delta^\epsilon}\leq C(\mathcal{E}_L+\|\cdot\|_2^2)\ \mbox{for some } 0<\epsilon< 1,\ C>0,
\end{eqnarray*}
or (\ref{distcompare})
\begin{eqnarray*}
d_L\leq C(d_\Delta)^c \ \mbox{for some } 0<c\leq 1,\ C>0
\end{eqnarray*}
holds, then $L$\ satisfies {\em (CK$0^+$)}.
\end{corollary}
\begin{proof}
Suppose (\ref{formcomparison}) holds. By \cite[Theorem 3.1]{Nashfractional}, because $H_t^\Delta$\ satisfies (CK$\lambda$) for arbitrarily small $\lambda>0$, by picking $\lambda$\ small enough, the condition $\epsilon>\frac{\lambda}{1+\lambda}$\ is met and $\Delta^{\epsilon}$\ satisfies the Nash inequality (\ref{Nashineq}). Then by the same argument as the beginning of the proof of Theorem \ref{gaussianestimatethm}, $L$\ satisfies the local Nash inequality, and hence the ultracontractivity condition (CK$\gamma$) for $0<t<1$\ with
\begin{eqnarray*}
\gamma=\frac{\lambda}{\epsilon+\epsilon\lambda-\lambda}.
\end{eqnarray*}
Letting $\lambda\rightarrow 0$\ then gives (CK$\gamma$) for arbitrarily small $\gamma>0$. So $L$\ satisfies (CK$0^+$).

Next suppose the distance comparison (\ref{distcompare}) holds. By Theorem \ref{formthm}, (\ref{formcomparison}) is satisfied for some $0<\epsilon<c\leq 1$, so we conclude that $L$\ satisfies (CK$0^+$).
\end{proof}

\begin{remark}
  If desired, one can derive more quantitative versions of Corollary \ref{cor-CK0} where the hypothesis involves a particular behavior of $\log \mu^{\Delta}_t(e)$ when $t$ tends to zero (i.e., $\log \mu^{\Delta}_t(e)\le C(\log 1/t)^\lambda$ for some $\lambda>1$ when $t$ goes to $0$), and the conclusion is more precise than saying that $\mu^L_t$ satisfies  (CK$0^+$).  \end{remark}

\section{Examples}
\label{examplesection}
Recall that a (standard) Milnor basis for the Lie algebra of $SU(2)$ is any basis $\{E_1,E_2,E_3\}$
such that $[E_1,E_2]=E_3, [E_2,E_3]=E_1$ and $[E_3,E_1]=E_2$.
Let $G=\prod_{i=1}^{\infty} SU(2)$. Let $\{E_i^1,E_i^2,E_i^3\}$ be a Milnor basis on the $i$th factor $SU(2)$ , $i\in \mathbb{N}_+$. For sequences of positive numbers $(a_i)_{i=1}^\infty$, $(b_i)_{i=1}^\infty$, $(c_i)_{i=1}^\infty$\ that tend to infinity as $i\rightarrow \infty$, let
\begin{eqnarray*}
\Delta&:=&-\sum_{i=1}^{\infty} a_i\left((E_i^1)^2+(E_i^2)^2+(E_i^3)^2\right)=:\sum_{i=1}^{\infty} a_i\Delta_i;\\
L&:=&-\left(\sum_{i=1}^{\infty} b_i(X_i^1)^2 + \sum_{i=1}^{\infty} c_i(X_i^2)^2\right).
\end{eqnarray*}
Here $X_i^1:=E_i^1$, $X_i^2:=E_i^2+E_{i+1}^2$.
The distances and heat kernels associated with each bi-invariant Laplacian $\Delta_i$\ on $SU(2)$\ and with the bi-invariant Laplacians $\Delta=\sum_{i=1}^{\infty} a_i\Delta_i$\ on $G=\prod_{i=1}^{\infty} SU(2)$, are  well studied. We first briefly review some results that we use here. The operator $L$\ is more complicated than $\Delta$\ in two aspects. First, each $b_i(X_i^1)^2 + c_i(X_i^2)^2$\ is not an operator on the $i$th factor $SU(2)$\ because the vector field $X_i^2=E_i^2+E_{i+1}^2$\ lives on two successive factors. Second, $L$\ is only a sub-Laplacian as none of the $E_i^3$\ directions are included. Our goal is  to find some conditions on the coefficients $(a_i)_{i=1}^\infty$ and $(b_i)_{i=1}^\infty,(c_i)_{i=1}^\infty$, that guarantee a comparison between the intrinsic distances $d_\Delta$\ and $d_L$, and between Dirichlet forms associated with these operators. Once such comparisons are established, we will deduce upper-bounds for the heat kernel associated with $L$.

\paragraph{Review of the geometry of $SU(2)$\ associated with any bi-invariant Laplacian.}
Let $\{E_1,E_2,E_3\}$\ be a Milnor basis for $SU(2)$. Let $\Delta_{SU(2)}:=-\left((E_1)^2+(E_2)^2+(E_3)^2\right)$\ denote the canonical bi-invariant Laplacian on $SU(2)$. In classical Lie group theory there are two widely used coordinates - the so-called canonical coordinates of the first and second kinds, cf. \cite[Section 2.10]{varadarajan}. Namely, there is some $\eta>0$\ such that the following maps $\Phi$, $\Psi$\ are smooth bijections between a cube $Q_\eta:=(-\eta,\eta)^3$\ and neighborhoods $U_\eta$, $V_\eta$\ of the identity $e$\ in $SU(2)$, with $\Phi(0)=\Psi(0)=e$. The maps are given by
\begin{eqnarray*}
\Phi:Q_\eta&\rightarrow& U_\eta\subset SU(2)\\
(x_1,x_2,x_3) &\mapsto& \exp{\{x_1E_1+x_2E_2+x_3E_3\}};
\end{eqnarray*}
\begin{eqnarray*}
\Psi:Q_\eta&\rightarrow& V_\eta\subset SU(2)\\
(y_1,y_2,y_3) &\mapsto& \exp{\{y_1E_1\}}\exp{\{y_2E_2\}}\exp{\{y_3E_3\}}.
\end{eqnarray*}
Moreover, the following formula and estimate hold.
\begin{itemize}
\item The distance defined as the length of the shortest curve (geodesic) connecting two points, agrees with the intrinsic distance. See \cite[Sections 5.5.3, 5.5.4]{SturmGeom}.
\item The one parameter subgroup $\exp{\left\{s\left(x_1E_1+x_2E_2+x_3E_3\right)\right\}}$, $s\in\mathbb{R}$, is a local geodesic with constant speed $\sqrt{(x_1)^2+(x_2)^2+(x_3)^2}$.
\item The distance between $e$\ and $\Phi(x_1,x_2,x_3)=\exp{\{x_1E_1+x_2E_2+x_3E_3\}}$\ is \begin{eqnarray}
d_{\Delta_{SU(2)}}(\Phi(x_1,x_2,x_3))=\sqrt{(x_1)^2+(x_2)^2+(x_3)^2},\label{distformula1stkind}
\end{eqnarray}
for $(x_1,x_2,x_3)\in Q_\eta$.
\item There are constants $c,C>0$\ such that (here we write $d_\Delta$\ for $d_{\Delta_{SU(2)}}$)
\begin{eqnarray}
\label{distformula2ndkind}
cd_{\Delta}(\Psi(y_1,y_2,y_3))&\leq& \max_{j=1,2,3}{|y_j|}\leq Cd_{\Delta}(\Psi(y_1,y_2,y_3)),
\end{eqnarray}
for $(y_1,y_2,y_3)\in Q_\eta$.
\end{itemize}
If we scale the canonical Laplacian by a factor $a>0$, i.e., consider
\begin{eqnarray*}
\Delta_a:=-a\left((E_1)^2+(E_2)^2+(E_3)^2\right)=a\Delta_{SU(2)},
\end{eqnarray*}
then
\begin{eqnarray*}
d_{\Delta_a}=\frac{1}{\sqrt{a}}d_{\Delta_{SU(2)}}.
\end{eqnarray*}
Recall that the diameter of $G$\ with respect to a distance $d$\ is $\mbox{diam}_d(G)=\sup_{x,y\in G}{d(x,y)}$. If $d$\ is the intrinsic distance of some sub-Laplacian $L$, we simply write $\mbox{diam}_L(G)$. The diameter of $SU(2)$\ with respect to $d_{\Delta_{SU(2)}}$\ is
\begin{eqnarray*}
\mbox{diam}_{\Delta_{SU(2)}}(SU(2))=2\pi,
\end{eqnarray*}
cf. \cite[Section 2.5]{SU2}. It follows that any element $h\in SU(2)$\ can be written as a product of $N$\ elements $h_k\in U_\eta$\ (respectively, $h_k\in V_\eta$), $h=\prod_{k=1}^{N} h_k$, and $N$\ is uniformly bounded above. Say the ball $B(e,2r)=\{h\in G:d_{\Delta_{SU(2)}}(h)<2r\}$\ is contained in $U_\eta$\ and $V_\eta$, then $N<2\pi/r=:N_0$. 
\begin{remark}
\label{expremark}
On $SU(2)$, every element is of the form $e^{uX}$\ for some unit left-invariant vector field $X$\ and some $u\in [0, \mbox{diam}_{\Delta_{SU(2)}}(SU(2))]=[0,2\pi]$\ (recall that $SU(2)$\ is isometric to the sphere $S^3$\ in $\mathbb{R}^4$\ with radius $2$; on $S^3$, such exponentials represent the great circles). Here $X=x_1E_1+x_2E_2+x_3E_3$\ is unit (with respect to the distance $d_{\Delta_{SU(2)}}$) means $\sqrt{(x_1)^2+(x_2)^2+(x_3)^2}=1$.
\end{remark}

\paragraph{Review of the geometry of $G=\prod SU(2)$\ associated with a bi-invariant Laplacian $\Delta=\sum a_i\Delta_i$.}
Next we review some facts for the product space $G=\prod SU(2)$. Any element $g\in G$\ is a product $g=\prod_ig_i$\ of elements $g_i$\ in the $i$th factor $SU(2)$. Recall that we call $\Delta=\sum_i\Delta_{a_i}$. By Remark \ref{expremark}, for each $i$, with respect to the basis $\{\sqrt{a_i}E_i^1, \sqrt{a_i}E_i^2, \sqrt{a_i}E_i^3\}$,
\begin{eqnarray*}
g_i=e^{u_i\left(x_i^1\sqrt{a_i}E_i^1+x_i^2\sqrt{a_i}E_i^2+x_i^3\sqrt{a_i}E_i^3\right)},
\end{eqnarray*}
for some $u_i,x_i^j$\ satisfying
\begin{eqnarray}
&&0\leq u_i\leq \mbox{diam}_{\Delta_{a_i}}(SU(2))=\frac{2\pi}{\sqrt{a_i}},\label{uibound}\\
&&\sqrt{(x_i^1)^2+(x_i^2)^2+(x_i^3)^2}=1.\notag
\end{eqnarray}
Here the subscript $i$\ indicates that everything is on the $i$th factor $SU(2)$.

The condition for $\prod_{i=1}^{\infty} SU(2)$\ to have finite diameter is (cf. \cite{orange})
\begin{eqnarray}
\label{aicondition}
\sum_{i=1}^{\infty}\frac{1}{a_i}<\infty.
\end{eqnarray}
Under this condition, the following properties are satisfied.
\begin{itemize}
\item[(i)] The distance $d_\Delta$\ is continuous on $G$, and
\begin{eqnarray}
\label{Pythagorean}
(d_\Delta(g))^2=\sum_{i=1}^\infty (d_{\Delta_{a_i}}(g_i))^2
\end{eqnarray}
by the Pythagorean theorem.
\item[(ii)] Conditions (\ref{uibound}) and (\ref{aicondition}) imply that the sum of squares of $u_i$\ is finite, then each
$g_i$\ can be rewritten as
\begin{eqnarray*}
g_i=\exp{\left\{\left(\sum_{l=1}^{\infty}u_l^2\right)^{1/2}\frac{u_i}{(\sum_{l=1}^{\infty}u_l^2)^{1/2}} \left(x_i^1\sqrt{a_i}E_i^1+x_i^2\sqrt{a_i}E_i^2+x_i^3\sqrt{a_i}E_i^3\right)\right\}}.
\end{eqnarray*}
Setting
\begin{eqnarray*}
\beta_i^j:=\frac{u_ix_i^j}{(\sum_{l=1}^{\infty}u_l^2)^{1/2}},\ i\in\mathbb{N}_+,\ j=1,2,3,
\end{eqnarray*}
so that $\sum_{i=1}^\infty \left((\beta_i^1)^2+(\beta_i^2)^2+(\beta_i^3)^2\right)$ is $1$, we then have
\begin{eqnarray*}
g=\prod_{i=1}^{\infty}g_i=\prod_{i=1}^{\infty}\exp{\left\{\left(\sum_{l=1}^{\infty}u_l^2\right)^{1/2}(\beta_i^1\sqrt{a_i}E_i^1+\beta_i^2\sqrt{a_i}E_i^2+\beta_i^3\sqrt{a_i}E_i^3)\right\}}.
\end{eqnarray*}
\item[(iii)] The geodesic from $e$\ to $g$\ with unit speed is then
\begin{eqnarray*}
&&t\mapsto \exp{\left\{t\sum_{i=1}^{\infty}(\beta_i^1\sqrt{a_i}E_i^1+\beta_i^2\sqrt{a_i}E_i^2+\beta_i^3\sqrt{a_i}E_i^3)\right\}},\\
&&\mbox{for }0\leq t\leq \left(\sum_{l=1}^{\infty}u_l^2\right)^{1/2}\leq 2\pi\left(\sum_{l=1}^{\infty}a_l^{-1}\right)^{1/2}.
\end{eqnarray*}
Here we used the commutativity between different $SU(2)$\ factors.
\end{itemize}

In summary, assuming (\ref{aicondition}), any element $g\in G$\ is of the form $g=e^{tX}$\ for some unit direction $X$, and $d_\Delta(g)=t$. The distance given by geodesic and the intrinsic distance coincide.

\paragraph{Review of the heat kernel condition (CK$\lambda$) on $G=\prod SU(2)$.}
Let $\lambda_k$\ be the spectral gap of the Laplacian $\Delta_{a_k}=-a_k\left((E_k^1)^2+(E_k^2)^2+(E_k^3)^2\right)$\ on the $k$th factor $SU(2)$. Then $\lambda_k=a_k\lambda_0$, where $\lambda_0$\ denotes the spectral gap of $\Delta_{SU(2)}$. Let
\begin{eqnarray}
\label{countingfunction}
N(r):=3\sharp\{k:\lambda_k\leq r\},\ r>0.
\end{eqnarray}
Assume that
\begin{eqnarray*}
\lim_{r\rightarrow\infty}\frac{1}{r}\log{N(r)}=0.
\end{eqnarray*}
By \cite[Theorem 8.1(1)]{Elliptic}, this condition is equivalent to that $\mu_t^\Delta$\ admits a continuous density function $\mu_t^\Delta(\cdot)$. By \cite[Theorem 8.1(2)]{Elliptic} (see also \cite{centralJFA}), the heat kernel $\mu_t^\Delta(\cdot)$\ satisfies
\begin{eqnarray*}
\exp{\left\{C_1N\left(\frac{C_2}{t}\right)\right\}}\leq \mu_t^\Delta (e)\leq \exp{\left\{c_1\int_0^\infty \frac{N(\xi/t)e^{-c_2\xi}}{\xi}\,d\xi\right\}}
\end{eqnarray*}
for some constants $C_1,C_2,c_1,c_2>0$. If $N(r)\lesssim r^\lambda$\ (i.e., $N(r)\leq cr^\lambda$\ for some $c>0$), then the upper bound satisfies
\begin{eqnarray*}
\exp{\left\{c_1\int_0^\infty \frac{N(\xi/t)e^{-c_2\xi}}{\xi}\,d\xi\right\}}\leq \exp{\left\{c_1\int_0^\infty \zeta^{\lambda-1}e^{-c_2t\zeta}\,d\zeta\right\}}\leq \exp{\left\{\frac{C}{t^\lambda}\right\}}
\end{eqnarray*}
for some constant $C>0$, thus the Gaussian semigroup $(\mu_t^\Delta)_{t>0}$\ satisfies (CK$\lambda$). Conversely, if $(\mu_t^\Delta)_{t>0}$\ satisfies (CK$\lambda$), then the lower bound of $\mu_t^\Delta(e)$\ implies that $N(r)\lesssim r^\lambda$.

The property $N(r)\lesssim r^\lambda$\ holds for $a_k=k^{1/\lambda}$, $k\in \mathbb{N}_+$, for example. The property (CK$\lambda$) with $0<\lambda<1$\ implies (\ref{aicondition}), see \cite[Proposition 5.17]{orange}.

\paragraph{Estimates for $d_L(g)$.} We now develop the key arguments for the comparison of the two distances $d_\Delta$\ and $d_L$. Our conclusion is the following proposition.
\begin{proposition}
\label{example}
Let $G=\prod_{i=1}^\infty SU(2)$. 
Let $\{E_i^1,E_i^2,E_i^3\}$, $i\in\mathbb{N}_+$\ be copies of a Milnor basis on the $i$th factor $SU(2)$. Let $X_i^1:=E_i^1$, $X_i^2:=E_i^2+E_{i+1}^2$. Define
\begin{eqnarray*}
\Delta&:=&-\sum_{i=1}^{\infty} a_i\left((E_i^1)^2+(E_i^2)^2+(E_i^3)^2\right);\\
L&:=&-\left(\sum_{i=1}^{\infty} b_i(X_i^1)^2 + \sum_{i=1}^{\infty} c_i(X_i^2)^2\right),
\end{eqnarray*}
where $(a_i)_{i=1}^{\infty}$, $(b_i)_{i=1}^{\infty}$, $(c_i)_{i=1}^{\infty}$\ are sequences of positive real numbers that tend to infinity as $i$\ does. Let $(\mu_t^\Delta)_{t>0}$\ and $(\mu_t^L)_{t>0}$\ denote the Gaussian semigroups corresponding to $\Delta$\ and $L$, respectively.

(1) If the coefficients satisfy
\begin{eqnarray}
\label{coeffcondition1}
\lefteqn{\sup_{i\in\mathbb{N}_+}\left\{\frac{a_i}{b_i}\right\}<\infty,\ \sum_{i=1}^{\infty}\frac{a_i^{1/3}}{b_i^{2/3}}<\infty,\ 
\sum_{i=1}^{\infty}\frac{a_i^{1/5}}{b_i^{3/5}}<\infty,}\notag\\
&&\sum_{i=1}^{\infty}\frac{a_i^{1/3}}{c_i^{2/3}}<\infty,\ 
\sum_{i=1}^{\infty}\frac{a_i^{1/5}}{c_i^{3/5}}<\infty,
\end{eqnarray}
then the intrinsic distances $d_\Delta$\ and $d_L$\ associated with $\Delta$\ and $L$\ satisfy
\begin{eqnarray*}
d_L\leq C(d_\Delta)^{1/3}
\end{eqnarray*}
for some $C>0$.

(2) If $(\mu_t^\Delta)_{t>0}$\ satisfies (CK$\lambda$) for some $0<\lambda<(-9+\sqrt{109})/14\approx 0.1$, then for any $(b_i)_{i=1}^{\infty},(c_i)_{i=1}^{\infty}$\ satisfying
\begin{eqnarray}
\label{coeffcondition2}
\sup_{i\in\mathbb{N}_+}\left\{\frac{a_i}{b_i}\right\}<\infty,\ \sup_{i\in\mathbb{N}_+}\left\{\frac{a_i}{c_i}\right\}<\infty,
\end{eqnarray}
$(\mu_t^L)_{t>0}$\ satisfies (CK$\gamma$) for any $3\lambda(-7\lambda^2-9\lambda+1)^{-1}=:\gamma_\lambda<\gamma<\infty$. If furthermore $0<\lambda<(-6+\sqrt{43})/7\approx 0.08$, then $(\mu_t^L)_{t>0}$\ satisfies the Gaussian type upper bound that for any $\gamma_\lambda<\gamma<1$,
\begin{eqnarray*}
\mu_t^L(x)\leq \exp{\left\{\frac{C}{t^\gamma}-\frac{C'd_L^2(x)}{t}\right\}},\ \forall 0<t<1.
\end{eqnarray*}
\end{proposition}
For example, setting $a_i=i^{1/\lambda}$, $0<\lambda<1$.
\begin{itemize}
\item For $0<\lambda<1/3$, the coefficient condition (\ref{coeffcondition1}) amounts to $b_i,c_i\geq C a_i$\ for some $C>0$. In particular, one can take $b_i,c_i\sim a_i=i^{1/\lambda}$.
\item For $1/3\leq \lambda<1$, the choice of $b_i,c_i=i^{\epsilon+1/\lambda}$, where $\epsilon>(3\lambda-1)/2\lambda$, satisfies the coefficient condition (\ref{coeffcondition1}).
\end{itemize}
\begin{proof}
From the definitions of the associated intrinsic distances $d_\Delta$, $d_L$, it is straightforward to deduce that for any $i\in\mathbb{N}_+$,
\begin{eqnarray*}
d_L(e^{zX_i^1})&\leq&\frac{1}{\sqrt{b_i}}z,\ \mbox{for }0\leq z\leq 2\pi;\\
d_L(e^{zX_i^2})&\leq&\frac{1}{\sqrt{c_i}}z,\ \mbox{for }0\leq z\leq 2\pi;\\
d_\Delta(e^{zE_i^j})&=&\frac{1}{\sqrt{a_i}}z,\ \mbox{for }0\leq z\leq 2\pi,\ j=1,2,3.
\end{eqnarray*}
We now estimate $d_{L}(e^{zE_i^j})$, $j=1,2,3$. As $X_i^1=E_i^1$,
\begin{eqnarray}
\label{distanceE1}
d_{L}(e^{zE_i^1})\leq\frac{1}{\sqrt{b_i}}z,\ \mbox{for }0\leq z\leq 2\pi.
\end{eqnarray}
To estimate the other two terms we use the following explicit formula proved in \cite[Theorem 5.10]{U2}. For any Milnor basis $\{E_1,E_2,E_3\}$\ on $SU(2)$, for any $s,t\in\mathbb{R}$,
\begin{eqnarray*}
e^{f(s,t)E_3}=e^{-\tau(s,t) E_2}e^{\frac{s}{2}E_1}e^{tE_2}e^{-sE_1}e^{-tE_2}e^{\frac{s}{2}E_1}e^{\tau(s,t) E_2},
\end{eqnarray*}
where $\tau,f$\ are continuous functions in explicit formulas satisfying $|\tau(s,t)|\leq |t|/2$, and for some small constant $c>0$\ (e.g. $c=1/4$), $|f(s,t)|\geq c|st|$\ for $|s|,|t|\leq \pi$.

Observe that for each $i\in\mathbb{N}_+$, $E_{i+1}^2$\ commutes with $X_i^1=E_i^1$\ and $X_i^2=E_i^2+E_{i+1}^2$. For any $0<z<c\pi^2$, let $0<|\xi|<\sqrt{|z|/c}<\pi$\ be such that $f(\xi,\xi)=z$\ (one can check from the formula $f(\xi,\xi)=2\arccos{(\cos^2{(\xi/2)}+\sin^2{(\xi/2)}\cos{\xi})}$\ that it is a strictly increasing bijection from $[0,\pi]$\ to $[0,2\pi]$). We have
\begin{eqnarray*}
e^{zE_i^3}=e^{f(\xi,\xi)E_i^3}=e^{-\tau(\xi,\xi) X_i^2}e^{\frac{\xi}{2}X_i^1}e^{\xi X_i^2}e^{-\xi X_i^1}e^{-\xi X_i^2}e^{\frac{\xi}{2}X_i^1}e^{\tau(\xi,\xi) X_i^2}.
\end{eqnarray*}
For $-c\pi^2<z<0$, we take the inverse of all the elements above to get
\begin{eqnarray*}
e^{zE_i^3}=e^{-|z|E_i^3}=e^{-f(\xi,\xi)E_i^3}=e^{-\tau(\xi,\xi) X_i^2}e^{-\frac{\xi}{2}X_i^1}e^{\xi X_i^2}e^{\xi X_i^1}e^{-\xi X_i^2}e^{-\frac{\xi}{2}X_i^1}e^{\tau(\xi,\xi) X_i^2}.
\end{eqnarray*}
It follows that for any $0<|z|<c\pi^2$,
\begin{eqnarray}
\lefteqn{d_L(e^{zE_i^3}) \leq \frac{1}{\sqrt{c_i}}|\tau(\xi,\xi)|+\frac{1}{\sqrt{b_i}}\left|\frac{\xi}{2}\right|}\notag\\
&+& \frac{1}{\sqrt{c_i}}|\xi|+\frac{1}{\sqrt{b_i}}|-\xi|+\frac{1}{\sqrt{c_i}}|-\xi|+\frac{1}{\sqrt{b_i}}\left|\frac{\xi}{2}\right|+\frac{1}{\sqrt{c_i}}|\tau(\xi,\xi)|\notag\\
&\leq& C\left(\frac{1}{\sqrt{b_i}}+\frac{1}{\sqrt{c_i}}\right)|\xi|\leq C'\left(\frac{1}{\sqrt{b_i}}+\frac{1}{\sqrt{c_i}}\right)\sqrt{|z|}. \label{distanceE2}
\end{eqnarray}
Here $C,C'$\ are independent of $i$. Lastly, as $X_i^1=E_i^1$\ and $E_i^2=[E_i^3,X_i^1]$,
\begin{eqnarray*}
e^{f(p,q)E_i^2}=e^{-\tau(p,q) X_i^1}e^{\frac{p}{2}E_i^3}e^{qX_i^1}e^{-pE_i^3}e^{-qX_i^1}e^{\frac{p}{2}E_i^3}e^{\tau(p,q) X_i^1}.
\end{eqnarray*}
For any $0< z< c\pi^2$, let $p=\xi^{4/3}$, $q=\xi^{2/3}$\ for some $0<|\xi|<\sqrt{|z|/c}$ such that $f(p,q)=z$. Here $f(\xi^{4/3},\xi^{2/3})=2\arccos{\left(\cos^2{(\xi^{4/3}/2)}+\sin^2{(\xi^{4/3}/2)}\cos{(\xi^{2/3})}\right)}$\ is a strictly increasing bijection from $[0,2.57]$\ to $[0,3.66]$\ roughly. For $-c\pi^2<z<0$\ we take the inverse of the expressions as above. Using $|\tau(p,q)|\leq |q|/2$\ and applying (\ref{distanceE2}) to estimate $d_L$\ of the terms involving $E_i^3$, we get that for any $0<|z|<c\pi^2$,
\begin{eqnarray}
\lefteqn{d_L(e^{zE_i^2})\leq C\left(\frac{q}{\sqrt{b_i}}+\left(\frac{1}{\sqrt{b_i}}+\frac{1}{\sqrt{c_i}}\right)\sqrt{p}\right)}\notag\\
&\leq& C'\left(\frac{1}{\sqrt{b_i}}+\frac{1}{\sqrt{c_i}}\right)\xi^{2/3}\leq C''\left(\frac{1}{\sqrt{b_i}}+\frac{1}{\sqrt{c_i}}\right)|z|^{1/3}. \label{distanceE3}
\end{eqnarray}
All constants $C,C',C''$\ are independent of $i$.

For $d_L(g)$\ of a general element $g\in G$, recall that $g=\prod_{i=1}^{\infty} g_i$\ with $g_i$\ in the $i$th $SU(2)$. Recall that any element $h\in SU(2)$\ can be written in the form $h=\prod_{j=1}^{N}h_j$\ where $h_j\in V_\eta$, $d_{\Delta_{SU(2)}}(h_j)<\eta/c$\ (as in (\ref{distformula2ndkind})), $N\leq N_0$. See the paragraph above Remark \ref{expremark}. Applying this to our example, for each $i\in\mathbb{N}_+$,
\begin{eqnarray*}
g_i=\prod_{k=1}^{N_i} h_{i,k}=\prod_{k=1}^{N_i} \exp{\{y_{i,k}^1E_i^1\}}\exp{\{y_{i,k}^2E_i^2\}}\exp{\{y_{i,k}^3E_i^3\}}.
\end{eqnarray*}
All subscripts $i$\ in this formula refer to the $i$th factor $SU(2)$. The second equality corresponds to the use of coordinates of the second kind. We may take all $N_i=N_0$\ by inserting identity elements in the product if necessary. By (\ref{distformula2ndkind}),
\begin{eqnarray}
\max_{\substack{1\leq k\leq N_i\\j=1,2,3}}{|y_{i,k}^j|}\leq Cd_{\Delta_{SU(2)}}(h_{i,k})=C\sqrt{a_i}d_{\Delta_{a_i}}(h_{i,k}),\label{estimate1}
\end{eqnarray}
further note that $d_{\Delta_{a_i}}(h_{i,k})\leq d_{\Delta_{a_i}}(g_i)$.

Note that different $SU(2)$\ factors commute, so 
\begin{eqnarray*}
\lefteqn{g=\prod_{i=1}^{\infty} g_i=\prod_{i=1}^{\infty}\prod_{k=1}^{N_0} h_{i,k}=\prod_{k=1}^{N_0}\left(\prod_{i=1}^{\infty} h_{i,k}\right)}\\
&=&\prod_{k=1}^{N_0}\left(\prod_{i=1}^{\infty} \exp{\{y_{i,k}^1E_i^1\}}\exp{\{y_{i,k}^2E_i^2\}}\exp{\{y_{i,k}^3E_i^3\}}\right)\\
&=&\prod_{k=1}^{N_0}\left\{\left(\prod_{i=1}^{\infty} \exp{\{y_{i,k}^1E_i^1\}}\right)\left(\prod_{i=1}^{\infty} \exp{\{y_{i,k}^2E_i^2\}}\exp{\{y_{i,k}^3E_i^3\}}\right)\right\}
\end{eqnarray*}
To see the last line, note that in the parenthesis of the middle line (i.e., the product over $i$\ with fixed $k$), all terms in front of each $\exp{\{y_{i,k}^1E_i^1\}}$\ are elements in other $SU(2)$ factors and commute with $\exp{\{y_{i,k}^1E_i^1\}}$. Moving $\exp{\{y_{i,k}^1E_i^1\}}$\ to the front gives $\left(\prod_{i=1}^{\infty} \exp{\{y_{i,k}^1E_i^1\}}\right)$\ in the last line. It follows that
\begin{eqnarray*}
\lefteqn{d_L(g)\leq \sum_{k=1}^{N_0}\left\{d_L\left(\exp{\left\{\sum_{i=1}^{\infty}y_{i,k}^1E_i^1\right\}}\right)+d_L\left(\prod_{i=1}^{\infty} \exp{\{y_{i,k}^2E_i^2\}}\exp{\{y_{i,k}^3E_i^3\}}\right)\right\}}\\
&\leq& \sum_{k=1}^{N_0}\left\{\left(\sum_{i=1}^{\infty}\frac{|y_{i,k}^1|^2}{b_i}\right)^{1/2}+\sum_{i=1}^{\infty}d_L(\exp{\{y_{i,k}^2E_i^2\}})+d_L(\exp{\{y_{i,k}^3E_i^3\}})\right\}.
\end{eqnarray*}
Using (\ref{distanceE2}) and (\ref{distanceE3}), together with (\ref{estimate1}), gives
\begin{eqnarray*}
\lefteqn{d_L(g)\leq N_0\max_{1\leq k\leq N_0} \left\{\left(\sum_{i=1}^{\infty}\frac{|y_{i,k}^1|^2}{b_i}\right)^{1/2}+\right.}\\
&&\left.\sum_{i=1}^{\infty} \left(\frac{1}{\sqrt{b_i}}+\frac{1}{\sqrt{c_i}}\right)\left((|y_{i,k}^2|)^{1/3}+(|y_{i,k}^3|)^{1/2}\right)\right\}\\
&\leq& CN_0\max_{1\leq k\leq N_0} \left\{\left(\sum_{i=1}^{\infty}\frac{a_i(d_{\Delta_{a_i}}(g_i))^2}{b_i}\right)^{1/2}+\right.\\
&&\left.\sum_{i=1}^{\infty} \left(\frac{1}{\sqrt{b_i}}+\frac{1}{\sqrt{c_i}}\right)\left((\sqrt{a_i}d_{\Delta_{a_i}}(g_i))^{1/3}+(\sqrt{a_i}d_{\Delta_{a_i}}(g_i))^{1/2}\right)\right\}.
\end{eqnarray*}
The first summation has upper bound
\begin{eqnarray*}
\sum_{i=1}^{\infty}\frac{a_i(d_{\Delta_{a_i}}(g_i))^2}{b_i}\leq \sup_{i\in\mathbb{N}_+}\left\{\frac{a_i}{b_i}\right\}\sum_{i=1}^{\infty}(d_{\Delta_{a_i}}(g_i))^2.
\end{eqnarray*}
Applying the Cauchy-Schwartz inequality for the second summation, we have
\begin{eqnarray*}
\lefteqn{\hspace{-0.5in}\sum_{i=1}^{\infty} \left(\frac{1}{\sqrt{b_i}}+\frac{1}{\sqrt{c_i}}\right)(\sqrt{a_i}d_{\Delta_{a_i}}(g_i))^{1/3}+\sum_{i=1}^{\infty} \left(\frac{1}{\sqrt{b_i}}+\frac{1}{\sqrt{c_i}}\right)(\sqrt{a_i}d_{\Delta_{a_i}}(g_i))^{1/2}}\\
&\hspace{-0.2in}\leq& \left(\sum_{i=1}^{\infty}a_i^{1/5}\left(\frac{1}{\sqrt{b_i}}+\frac{1}{\sqrt{c_i}}\right)^{6/5}\right)^{5/6}\left(\sum_{i=1}^{\infty} d_{\Delta_{a_i}}(g_i)^2\right)^{1/6}\\
&&+\left(\sum_{i=1}^{\infty}a_i^{1/3}\left(\frac{1}{\sqrt{b_i}}+\frac{1}{\sqrt{c_i}}\right)^{4/3}\right)^{3/4}\left(\sum_{i=1}^{\infty}d_{\Delta_{a_i}}(g_i)^2\right)^{1/4}.
\end{eqnarray*}

Note that $\sum_i d_{\Delta_{a_i}}(g_i)^2=(d_\Delta(g))^2$\ as reviewed in (\ref{Pythagorean}). If the ratios $a_i/b_i$\ are bounded uniformly, and the coefficient summations converge, we conclude that there is some constant $C>0$\ such that
\begin{eqnarray*}
d_L(g)\leq Cd_{\Delta}(g)^{1/3},\ \forall g\in G=\prod_{i=1}^{\infty}SU(2).
\end{eqnarray*}
Breaking the mixed terms like
\begin{eqnarray*}
\lefteqn{\left(\frac{1}{\sqrt{b_i}}+\frac{1}{\sqrt{c_i}}\right)^{4/3}\leq 2^{4/3}\left(\max{\left\{\frac{1}{\sqrt{b_i}},\frac{1}{\sqrt{c_i}}\right\}}\right)^{4/3}}\\
&=& 2^{4/3}\max{\left\{(b_i)^{-2/3},(c_i)^{-2/3}\right\}}\leq 2^{4/3}\left(\left(\frac{1}{b_i}\right)^{2/3}+\left(\frac{1}{c_i}\right)^{2/3}\right)
\end{eqnarray*}
so that the coefficient summations are in $a_i,b_i$\ and in $a_i,c_i$\ separately, we obtain the condition (\ref{coeffcondition1}). 

Next we discuss more concrete conditions on $b_i,c_i$\ that guarantee (\ref{coeffcondition1}), under different summability properties of $(a_i)_{i=1}^{\infty}$. Write $b_i:=a_i\kappa(i)$\ for $\kappa(i)>0$, $i\in\mathbb{N}_+$. Then the two summations in the first line of (\ref{coeffcondition1}) are
\begin{eqnarray}
\sum_{i=1}^{\infty} \frac{a_i^{1/3}}{b_i^{2/3}}&=&\sum_{i=1}^{\infty} \frac{1}{a_i^{1/3}\kappa(i)^{2/3}},\label{coeffintermediate1}\\ \sum_{i=1}^{\infty} \frac{a_i^{1/5}}{b_i^{3/5}}&=&\sum_{i=1}^{\infty} \frac{1}{a_i^{2/5}\kappa(i)^{3/5}}.\label{coeffintermediate2}
\end{eqnarray}
We consider three cases of $(a_i)_{i=1}^{\infty}$.\\
Case 1. Suppose $\sum_i a_i^{-1/3}<\infty$, which implies $\sum_i a_i^{-2/5}\leq \left(\sum_i a_i^{-1/3}\right)^{6/5}<\infty$. Taking $\kappa(i)$\ away from $0$, i.e.,
\begin{eqnarray*}
\sup_{i\in\mathbb{N}_+} \kappa(i)^{-1}<\infty,
\end{eqnarray*}
guarantees that the summations in (\ref{coeffintermediate1}) and (\ref{coeffintermediate2}) converge.\\
Case 2. Suppose $\sum_i a_i^{-r}<\infty$\ for some $1/3<r\leq 2/5$. Applying H\"{o}lder inequality to (\ref{coeffintermediate1}), we have for $p>1$\ with $p/3=r$,
\begin{eqnarray*}
\sum_{i=1}^{\infty} \frac{a_i^{1/3}}{b_i^{2/3}}=\sum_{i=1}^{\infty} \frac{1}{a_i^{1/3}\kappa(i)^{2/3}}\leq \left(\sum_{i=1}^{\infty} \frac{1}{a_i^{p/3}}\right)^{1/p}\left(\sum_{i=1}^{\infty}\frac{1}{\kappa(i)^{2p/3(p-1)}}\right)^{1-1/p}.
\end{eqnarray*}
By assumption, the first factor of the product is finite. As $p=3r$, $2p/3(p-1)=2r/(3r-1)$, so the condition
\begin{eqnarray}
\label{kappa1}
\sum_{i=1}^{\infty}\frac{1}{\kappa(i)^{2r/(3r-1)}}<\infty
\end{eqnarray}
implies the convergence of (\ref{coeffintermediate1}). As $r\leq 2/5$, $\sum_i a_i^{-2/5}<\infty$, so (\ref{kappa1}) (in particular, $\kappa(i)\rightarrow\infty$) guarantees the convergence of (\ref{coeffintermediate2}).\\
Case 3. Finally, suppose $\sum_i a_i^{-r}<\infty$\ for some $2/5<r\leq 1$. Recall that the condition $\sum_i a_i^{-1}<\infty$\ is equivalent to $d_\Delta$\ being finite, so here we require $r\leq 1$. Condition (\ref{kappa1}) implies the convergence of (\ref{coeffintermediate1}) as before. The same type of argument applied to (\ref{coeffintermediate2}) gives
\begin{eqnarray*}
\sum_{i=1}^{\infty} \frac{a_i^{1/5}}{b_i^{3/5}}=\sum_{i=1}^{\infty} \frac{1}{a_i^{2/5}\kappa(i)^{3/5}}\leq \left(\sum_{i=1}^{\infty} \frac{1}{a_i^{2q/5}}\right)^{1/q}\left(\sum_{i=1}^{\infty}\frac{1}{\kappa(i)^{3q/5(q-1)}}\right)^{1-1/q},
\end{eqnarray*}
where $2q/5=r$, $q=5r/2>1$, and the condition
\begin{eqnarray}
\label{kappa2}
\sum_{i=1}^{\infty}\frac{1}{\kappa(i)^{3r/(5r-2)}}<\infty
\end{eqnarray}
implies the convergence of (\ref{coeffintermediate2}). Comparing the powers of (\ref{kappa1}) and (\ref{kappa2}), we see that
\begin{eqnarray*}
\frac{2r}{3r-1}-\frac{3r}{5r-2}=\frac{r(r-1)}{(3r-1)(5r-2)}\leq 0,
\end{eqnarray*}
so (\ref{kappa1}) is a stronger requirement than (\ref{kappa2}).

The same arguments work for the coefficient summations in $a_i,c_i$\ in (\ref{coeffcondition1}). In summary,
\begin{eqnarray*}
&&\mbox{(i) if }\sum_{i=1}^\infty \frac{1}{a_i^{1/3}}<\infty,\ \mbox{then }\sup_{i\in\mathbb{N}_+}\left\{\frac{a_i}{b_i}\right\}<\infty,\ \sup_{i\in\mathbb{N}_+}\left\{\frac{a_i}{c_i}\right\}<\infty\ \mbox{imply (\ref{coeffcondition1})};\\
&& \mbox{(ii) if }\sum_{i=1}^\infty \frac{1}{a_i^{r}}<\infty\ \mbox{for some }\frac{1}{3}<r\leq 1,\ \mbox{then }\sum_{i=1}^\infty \left(\frac{a_i}{b_i}\right)^{2r/(3r-1)}<\infty,\\
&&\sum_{i=1}^\infty \left(\frac{a_i}{c_i}\right)^{2r/(3r-1)}<\infty\ \mbox{imply (\ref{coeffcondition1})}.
\end{eqnarray*}

If in addition the coefficients $(a_i)_{i=1}^{\infty}$\ are chosen so that $\mu_t^\Delta$\ satisfies (CK$\lambda$) for some $0<\lambda<1$, for example, $a_i=i^{1/\lambda}$, then Theorem \ref{formthm} and Theorem \ref{gaussianestimatethm} apply and provide comparisons between Dirichlet forms and upper bound estimates for $\mu_t^L$. Letting $c=1/3$\ in Theorem \ref{gaussianestimatethm} gives the upper bounds of $\lambda$\ in Item (2) of Proposition \ref{example}. In particular, the upper bounds $0.1$, $0.08$\ are less than $1/3$. Note that the condition (CK$\lambda$) implies that the counting function $N(s)$\ for $(a_i)_{i=1}^{\infty}$, defined in (\ref{countingfunction}), satisfies $N(s)\lesssim s^{\lambda}$. Then (say $a_1\geq 1$)
\begin{eqnarray*}
\sum_{i=1}^\infty \frac{1}{a_i^{1/3}}=\int_1^\infty \frac{1}{s^{1/3}}\,dN(s)\lesssim \frac{s^\lambda}{s^{1/3}}\bigg|_1^\infty+\frac{1}{3}\int_1^\infty s^\lambda s^{-4/3}\,ds=-1+\frac{1}{3}\frac{1}{\frac{1}{3}-\lambda}<\infty.
\end{eqnarray*}
So in this case, (\ref{coeffcondition2}) guarantees (\ref{coeffcondition1}) by Item (i) in the previous paragraph.
\end{proof}

\bibliographystyle{plain}

\end{document}